\documentclass[11pt, a4paper]{amsart}

\usepackage{latexsym}

\usepackage{color}

\usepackage{amssymb}
\usepackage[centertags]{amsmath}
\usepackage{verbatim}

\usepackage{units}

\usepackage{amsthm}

\newtheorem{theorem}{Theorem}

\newtheorem{lemma}{Lemma}

\newtheorem{claim}[equation]{Claim}

\theoremstyle{definition}

\newcommand{\theoremname}{testing}
\theoremstyle{remark}
\newtheorem*{remark*}{Remark}

\numberwithin{equation}{section}

\renewcommand{\phi}{\varphi}

\newcommand{\e}{\varepsilon}

\newcommand{\la}{\lambda}

\newcommand{\cE}{\mathcal E}

\newcommand{\cP}{\mathcal P}

\newcommand*{\dprime}{{\prime \prime}}
\newcommand*{\lapl}{\Delta}
\newcommand*{\dd}[1]{{\rm d}#1}
\newcommand*{\avg}[1]{\left\langle {#1} \right\rangle}
\newcommand*{\wbar}[1]{\overline{#1}}
\newcommand*{\what}[1]{\widehat{#1}}
\newcommand*{\indf}[1]{{1\hskip-2.5pt{\rm l}}_{#1}}

\def\eqdef{\overset{\rm{def}}{=}}

\renewcommand{\le}{\leqslant}
\renewcommand{\ge}{\geqslant}

\newcommand{\bC}{\mathbb C}

\newcommand{\bD}{\mathbb D}
\newcommand{\bT}{\mathbb T}

\title
{Distribution of zeroes of Rademacher Taylor series}

\author{Fedor Nazarov}
\address{Department of Mathematical Sciences, Kent State University, Kent OH 44242, USA
}
\email{nazarov@math.kent.edu}

\author{Alon Nishry}
\address{Department of Mathematics, University of Michigan, Ann Arbor, USA
}
\email{alonnish@umich.edu}

\author{Mikhail Sodin}
\address{School of Mathematical Sciences\\
Tel Aviv University\\
Tel Aviv 69978\\
Israel}
\email{sodin@post.tau.ac.il}
\thanks{This work was partially supported by Grant No.~2012037 of the United States - Israel
Binational Science Foundation (F.N., M.S.), by U.S. National Science Foundation Grants
DMS-0800243 (F.N.) and DMS-1128155 (A.N.), and by Grant No.~166/11  of the Israel Science Foundation of the
Israel Academy of Sciences and Humanities (A.N., M.S.)}

\begin{document}

\begin{abstract}

We find the asymptotics of the counting function of zeroes of random entire functions
represented by Rademacher Taylor series. We also give the asymptotics of the weighted counting
function, which takes into account the arguments of zeroes. These results answer several questions
left open after the pioneering work of Littlewood and Offord of 1948.

The proofs are based on our recent result on the logarithmic integrability of Rademacher Fourier series.

\end{abstract}

\maketitle

\section{Introduction and main results}

In this work, we consider the zero distribution of random entire functions represented by
the Rademacher Taylor series
\[
F(z) = \sum_{k\ge 0} \xi_k a_k z^k\,,
\]
where $\xi_k$ are independent Rademacher (a.k.a. Bernoulli) random variables, which take the values
$\pm 1$ with probability $\tfrac12$ each, and $\{ a_k\}$ is a (non-random)
sequence of complex numbers such that $ \lim_k |a_k|^{1/k} = 0 $ and
$ \# \{ k\colon a_k \ne 0 \} = \infty $.

\subsection{Peculiarity of the Rademacher case. R\^{o}le of the logarithmic integrability}

Consider a more general class of random Taylor series
with infinite radius of convergence:
\[
F(z) = \sum_{k\ge 0} \chi_k a_k z^k\,,
\]
in which the Rademacher random variables $\xi_k$ are replaced with general
independent identically distributed mean zero complex-valued random variables $\chi_k$ normalized by the condition $\cE |\chi_k|^2=1$,
and $\{ a_k\}$ are as above.
Let $ Z_F$ be the zero set of $F$ (with multiplicities).
Let us try to figure out how the asymptotics of the random counting
function $ n_F(r) = \# \{\zeta\in Z_F\colon |\zeta|\le r \}$ should look as $r\to\infty$.

Put
\[
\sigma_F(r)^2 = \mathcal E\bigl\{ |F(z)|^2 \bigr\} =
\sum_{k \ge 0} |a_k|^2 r^{2k}\,.
\]
To simplify the exposition, assume that $|a_0|=1$.
Denote by
\[
N_F(r) = \int_0^r \frac{n_F(t)}{t}\, \dd{t}
\]
the integrated counting function of the zero set $Z_F$. Then, by Jensen's formula,
\begin{eqnarray*}
N_F(r) &=& \int_{-\pi}^\pi \log|F(re^{{\rm i}\theta})|\, \frac{\dd{\theta}}{2\pi} - \log|F(0)| \\
&=& \log\sigma_F(r) + \int_{-\pi}^\pi \log|\what{F}_r(\theta)|\, \frac{\dd{\theta}}{2\pi} - \log|\chi_0|\,, \nonumber
\end{eqnarray*}
where $ \what{F}_r(\theta) \stackrel{\rm def}= F(re^{{\rm i}\theta})/\sigma_F(r) $. Note that
$ \what{F}_r(\theta) = \sum_{k\ge 0}\, \chi_k
\what{a}_k(r) e^{ {\rm i}k\theta} $
is a random Fourier series satisfying the condition
$\sum_{k\ge 0} | \what{a}_k(r) |^2 = 1$.

\medskip
First, assume that the
$\chi_k$'s are standard complex-valued Gaussian random variables. Then for every $\theta$,
the random variable $ \what{F}_r(\theta) $ is again a standard complex-valued Gaussian
random variable, and
$ \cE \bigl| \log |\what{F}_r(\theta)| \bigr| $ is a positive numerical constant.
Therefore,
\begin{equation}\label{eq:N-estimate}
\sup_{r>0}\, \cE \bigl| N_F(r) - \log\sigma_F(r) \bigr| \le C\,.
\end{equation}
Since both $ N_F(r) $ and $ \log \sigma_F(r) $ are convex functions of $\log r$, we can derive
from here that the functions
\[
n_F(r) = \frac{\dd{N_F(r)}}{\dd{\log r}} \quad \text{and} \quad
s_F(r)=\frac{\dd{\log\sigma_F(r)}}{\dd{\log r}}
= \frac{\sum_{k\ge 1} k|a_k|^2 r^{2k}}{\sum_{k\ge 0} |a_k|^2 r^{2k}}
\]
are also close for most values of $r$.
If we are interested in the
angular distribution of zeroes, the same idea works, we only need to replace Jensen's
formula by its modification for angular sectors.

\medskip
The same approach works in the Steinhaus case when
$\chi_k=e^{2\pi {\rm i} \gamma_k}$, where $\gamma_k$ are independent and
uniformly distributed on $[0, 1]$. In this case, one needs to
estimate the expectation of the modulus of the logarithm of the absolute value
of a normalized linear combination of independent Steinhaus
variables. This was done by Offord in~\cite{Offord-1968}; twenty
years later, Ullrich~\cite{Ullrich-88a, Ullrich-88b} and Favorov~\cite{Favorov-90, Favorov-94}
independently rediscovered his idea and applied it to various other problems.
See also recent works by Mahola and Filevich~\cite{MaFi1, MaFi2}.

\medskip
A linear combination of Rademacher random variables $ x = \sum_k \xi_k a_k $ can vanish with
positive probability. This leaves no hope to get a uniform lower bound for the logarithmic expectation
$ \cE \bigl\{ \log |x | \bigr\} $.
In~\cite{L-O}, Littlewood
and Offord invented ingenious and formidable techniques to
circumvent this obstacle. These techniques were further
developed by Offord in~\cite{Offord-1965, Offord-1995}. Apparently, the
methods of these works were
not sufficiently powerful to arrive at the same
conclusions as for the Gaussian and the Steinhaus coefficients.
Still, note that in order to estimate the error
term in the Jensen formula we do not need to estimate
$\cE \bigl| \log|\what{F}_r(\theta)|\, \bigr| $ uniformly in $\theta$.
Instead, we will be using the estimate
\begin{equation}\label{Log_int_est}
\cE \Bigl\{ \int_{-\pi}^\pi \bigl| \log|\what{F}_r(\theta)|\, \bigr|^p\, \frac{\dd{\theta}}{2\pi}
\Bigr\} \le (Cp)^{6p}, \qquad p\ge 1\,,
\end{equation}
proven in our recent work on the logarithmic integrability of
Rademacher Fourier series~\cite[Corollary 1.2]{NNS}. This will allow us to
extend the results known for the Gaussian and the Steinhaus coefficients
to the Rademacher case.

\medskip
Now, we describe the main results of this work. In what follows, we will use the notation
$r\bD = \{z\colon |z| < r\}$, $r\bar\bD = \{z\colon |z| \le r\}$, and $r\bT = \{ z\colon |z|=r\}$.
By $(\Omega, \cP)$ we always denote our probability space.

\subsection{Asymptotics of the number of zeroes in disks of large radii}

First, we address the asymptotics of the random counting function
$ n_F(r) = \# \{\zeta\in Z_F \cap r\bar\bD \} $.
Our asymptotics will hold when $r$ tends to infinity outside an exceptional set $E\subset [1, \infty)$ of finite
logarithmic length:
\[
m_\ell (E) = \int_E \frac{\dd{t}}{t} < \infty\,.
\]
Note that if the sequence $\{|a_k|\}$ is very irregular, the counting function $n_F(r)$
may exhibit a fast growth on short intervals, so the introduction of the set $E$
is unavoidable.

\begin{theorem}\label{thm1}
There exists a set $ E \subset [1, \infty) $ {\rm (}depending on $|a_k|$ only{\rm )} of finite logarithmic length
such that

\smallskip\noindent {\rm (i)} for almost every $\omega\in\Omega$, there exists $r_0(\omega)\in [1, \infty)$
such that for  every $r\in [r_0(\omega), \infty)\setminus E $ and
every $ \gamma > \tfrac12 $,
\[
\bigl| n_F(r) - s_F(r) \bigr| \le C(\gamma) s_F(r)^\gamma\,;
\]

\noindent {\rm (ii)}
for every $r\in [1, \infty)\setminus E $, and every $\gamma > \tfrac12 $,
\[
\cE \bigl|  n_F(r) - s_F(r) \bigr| \le C(\gamma) s_F(r)^\gamma\,.
\]
\end{theorem}

\subsection{Angular distribution of zeroes}

To address the angular distribution of zeroes, we introduce the counting
function
\[
n_F(r, \phi) = \sum_{\zeta\in (Z_F\setminus\{0\})\cap r\bar\bD} \phi (\arg \zeta)\,.
\]
Here and below, $\phi$ is a $2\pi$-periodic $C^2$-function, $0 \le \phi \le 1$.

In what follows, we denote by $A_F$ various positive constants that may depend only on the sequence
$\{ |a_k| \}$ of the absolute values of the Taylor coefficients of $F$. The symbol $ \avg{h} $ will stand
for the mean
\[
\avg{h} = \int_{-\pi}^\pi h(\theta)\, \frac{\dd{\theta}}{2\pi}\,.
\]

\begin{theorem}\label{thm2}
There exists a set $E\subset [1, \infty)$ {\rm (}depending on $|a_k|$ only{\rm )} of finite logarithmic length  such that

\smallskip\noindent {\rm (i)} for almost every $\omega\in\Omega$, every $r\in [r_0(\omega), \infty)\setminus E$,
every $2\pi$-periodic $C^2$-smooth function $\phi\colon [-\pi, \pi] \to [0, 1]$,
every $\gamma>\tfrac12$, and every $q>1$,
\[
\bigl| n_F(r, \phi) - \cE\bigl\{ n_F(r, \phi) \bigr\} \bigr| \le C(\gamma, q) \left( 1+ \| \phi'' \|_q \right)
\left( s_F(r)^\gamma + \log^\gamma r + A_F \right);
\]

\noindent {\rm (ii)}
for every $r\in [1, \infty)\setminus E$, every $2\pi$-periodic $C^2$-smooth function $\phi\colon [-\pi, \pi] \to [0, 1]$, every $\gamma>\tfrac12$, and every $q>1$,
\[
\cE \bigl| n_F(r, \phi) - \avg{\phi} s_F (r) \bigr| \le C(\gamma, q) \left( 1+ \| \phi'' \|_q \right)
\left( s_F(r)^\gamma + \log r + A_F \right).
\]
\end{theorem}

Theorem~\ref{thm2} yields the angular equidistribution of zeroes of $F$ provided that $s_F(r)$ does not grow
too slowly:
\[
\lim_{r\to\infty} \frac{s_F(r)}{\log r} = +\infty\,.
\]
Taking into account that $\log\sigma_F (r)$ is a convex function of $\log r$,
it is not difficult to see that this condition is equivalent to
\[
\lim_{r\to\infty} \frac{\log\sigma_F(r)}{\log^2 r} = +\infty\,,
\]
which in turn is equivalent to a more customary growth condition:
\[
\lim_{r\to\infty} \frac{\log M_F(r)}{\log^2 r} = +\infty\,, \qquad M_F(r) = \max_{r\bar D} |F|\,,
\]
which often occurs in the theory of entire functions, cf.~\cite[Section~7.2]{Hayman}.

It is also worth mentioning that
the first statement of Theorem~\ref{thm2} remains meaningful as long as $s_F(r)>\log^\kappa r$ with some
$\kappa > \tfrac12 $; i.e., beyond the $\log r$-threshold.

\subsection{Relation of our results to those by Littlewood and Offord}

In~\cite{L-O}, Littlewood and Offord  studied the distribution of zeroes of random entire functions of finite positive
order represented by Rademacher Taylor series. They used the maximal term $\mu _F(r) = \max_{k\ge 0} \left( |a_k|r^k \right)$ of the Rademacher
Taylor series $F$, which is basically equivalent to the quantity $\sigma_F (r)$ we are using here:
obviously, $ \mu_F (r) \le \sigma_F(r) $ everywhere, while, for every $\gamma>\tfrac12$,
$ \sigma_F(r) \le \mu_F(r) \log^\gamma \mu_F(r) $ outside an exceptional set of $r$'s of
finite logarithmic length (this is  a  classical result of Wiman and Valiron, see, for example,~\cite[Section~6.2]{Hayman}).
Littlewood and Offord discovered that, for every $\varepsilon>0$,
\begin{equation}\label{eq:LO}
\log |F(re^{{\rm i}\theta})| \ge \log\mu_F(r) - O_\varepsilon (r^\varepsilon)
\end{equation}
everywhere in the complex plane outside a union of simply connected domains of
small diameters. They called these domains ``pits''. Littlewood and Offord provided
a very detailed information about the sizes of the pits and the distribution of their locations.
From this, they were able to obtain some upper and lower a.s. bounds for the random integrated counting functions
$ N_F(r)$ and $N_F(r, \phi)$. However, these bounds differed
by a positive constant factor and did not yield the leading term of the asymptotics.

Later, Offord~\cite{Offord-1965, Offord-1995} extended the main results of~\cite{L-O}
to random entire functions of positive or infinite order of growth
represented by random Taylor series with more or less arbitrarily distributed sequence of independent random coefficients.

\subsection{Regularly decaying sequences $\{|a_k|\}$}
If the sequence of absolute values $ \{ |a_k|\} $ behaves very regularly:
\[
|a_k| = (\Delta+o(1))^k e^{-\alpha k \log k}\,, \qquad k\to\infty\,,
\]
with some positive constants $\Delta$ and
$\alpha$, then combining ~\eqref{eq:LO} with some results from the Levin-Pfluger theory of entire functions of completely
regular growth, one can obtain the leading term of the asymptotics provided by Theorems~\ref{thm1} and~\ref{thm2}.
It is also worth mentioning that recently Kabluchko and Zaporozhets~\cite[Corollary~2.6]{KZ} found a new elegant
approach to this special case, which is based on estimates for the concentration function combined with some
tools from potential theory. Their approach works for a very general class of non-degenerate i.i.d.
random variables $\chi_k$ (it needs only that $\cE \bigl\{ \log^+|\chi_k| \bigr\} <\infty $).
However, it seems that their approach should not work when $|a_k|$ does not have a very regular behavior.

Yet another approach was recently developed by Borichev, Nishry and Sodin in~\cite{BNS}. That approach works for certain
correlated stationary sequences $\chi_k$ as well as for some pseudo-random sequences of arithmetic origin, but still
requires a high regularity of the non-random sequence $\{ |a_k| \}$.

\subsection{Series with dominating central terms}

We complete this introduction with a brief discussion of (deterministic) Taylor series
\[
F(z) = \sum_{k\ge 0} \e_k a_k z^k,  \qquad \e_k\in \{\pm 1\}
\]
in which each non-zero term dominates on some circumference centered at the origin, i.e.,
series such that for every $k$ with $a_k\ne 0$, there exists $r_k>0$ such that
\[
|a_k|r_k^k > K \sum_{\ell\colon \ell\ne k} |a_k| r_k^\ell \qquad {\rm with\ some\ } K\ge 1\,.
\]
Note that this condition does not depend on the choice of the signs $\e_k$, so the corresponding
central term $\e_k a_k z^k$ dominates in all series simultaneously and, by Rouch\'e's theorem,
\[
n_F(r_k) = k \qquad {\rm regardless\ of\ } \{\e_\ell\}\,.
\]
This can be used to check sharpness of our constructions.

\subsubsection{}
First, we can give each power $k$ a possibility to dominate, thus ensuring that each annulus
$A_k=\{z\colon r_k<|z|<r_{k+1}\}$ contains exactly one zero of $F$. If $K$ is sufficiently large,
then the sum $\e_k a_k z^k + \e_{k+1} a_{k+1}z^{k+1}$ dominates the rest of the series
in the whole annulus $A_k$ except a small angle where the arguments of the two terms
are nearly opposite. So we can guarantee that the argument of the unique zero of $F$ in $A_k$
is close to that of \[-\frac{\e_k}{\e_{k+1}}\, \frac{a_k}{a_{k+1}}\,.\]
Since the first factor is just $\pm 1$, we can create almost as irregular angular
distribution of arguments of zeroes as we want. For instance, if $a_k$'s are real, then all
zeroes of $F$ will be real as well. This does not contradict Theorem~\ref{thm2} because
giving each index a possibility to dominate imposes a severe restriction on the growth
of $f$ and, thereby, on the growth of $n_F(r)$. It turns out that in this ``totally
irregular angular distribution case'', we have $n_F(r)$ and $s_F(r)$ comparable to
$\log r$, so the error term in part (ii) of Theorem~\ref{thm2} starts to exceed
the main one.

\subsubsection{}\label{subsub:lacunary}
Another possibility is to create a lacunary series
\[
F(z) = \sum_{j\ge 0} \e_j a_j z^{\la_j},  \qquad \e_j\in \{\pm 1\}
\]
in which the positive integer indices $\{\la_j\}$, $\la_0<\la_1< \ldots $, are sufficiently sparse.
In this case, there are sharp jumps in the number of zeroes of $F$ in narrow annuli
around the circumferences $C_j=\{z\colon |z|=\rho_j\}$ with radii given by
\[
\rho_j^{\la_{j+1}-\la_j} = \frac{a_{j}}{a_{j+1}}
\]
on which the subsequent non-zero terms of the series have equal
absolute value. On the other hand, the function $s_F(r)$, being defined by
a relatively nice formula, is necessarily rather smooth near the
radii $\rho_j$, so it starts growing somewhat earlier and finishes growing somewhat
later than $n_F(r)$. This creates large errors of opposite signs in the formula $n_F(r)\approx s_F(r)$
slightly to the left and slightly to the right of $\rho_j$, which shows that, in general,
allowing an exceptional set $E$ in Theorem~\ref{thm1} is inevitable.

\section{Preliminaries}

\subsection{Notation}
Throughout the paper we use the following notation:
\begin{itemize}
\item[$\diamond$]
For a function $h\colon [-\pi,\pi] \to \bC$, we write
\[
\avg{h} = \int_{-\pi}^\pi h(\theta) \, \frac{\dd{\theta}}{2 \pi} \qquad {\rm and} \quad
\| h \|_q = \left(\int_{-\pi}^\pi | h(\theta) |^q \, \frac{\dd{\theta}}{2 \pi}\right)^\frac1q\,.
\]

\medskip
\item[$\diamond$]
For a random variable $Y$ with finite first moment, we write $\overline{Y} = Y - \cE Y$.

\medskip
\item[$\diamond$]
By $F$ we denote a random entire function represented by a Rademacher Taylor series.

\medskip
\item[$\diamond$]
We denote the variance of $F(z)$ by
$\sigma_F(r)^2 = \cE\{ |F(z)|^2\}$, $r=|z|$, and put
$\what{F}_r(\theta)=F(re^{{\rm i}\theta})/\sigma_F(r)$.

\medskip
\item[$\diamond$]
We often use the notation
$ X_r = X_r(\theta) = \log | \what{F}_r(\theta) | $.

\medskip
\item[$\diamond$]
By $Z_F$ we denote the zero set of $F$.

\medskip
\item[$\diamond$]
By $C,c$ we denote various positive numerical constants. Their values may change from line to line.
If $\kappa$ is a parameter, then $C(\kappa), c(\kappa)$ are positive expressions that depend only on $\kappa$.

\medskip
\item[$\diamond$]
By $A_F$ we denote various positive expressions that may depend only on the sequence $\{|a_k|\}$ of the absolute values of the Taylor coefficients of $F$.
\end{itemize}

\subsection{Normalization}
When proving Theorems~\ref{thm1} and~\ref{thm2} we assume that
\[
F(z) =  1 + \sum_{k\ge 1} \xi_k a_k z^k
\]
with
\[
\sum_{k\ge 1} |a_k| \le \tfrac12
\]
(as before, $\lim_k |a_k|^{1/k}=0$, $\#\{k\colon a_k\ne 0\}=\infty$, and $\xi_k$ are
independent Rademacher random variables). To reduce the arbitrary Rademacher Taylor series
$F$ to this special form, first, we replace $F$ by the function
$F_1(z)=F(z)/(\xi_m a_m z^m)$, where $m$ is the least index with $a_m\ne 0$. For this function,
we have $n_{F_1}(r)=n_F(r)-m$, and $n_{F_1}(r, \phi) = n_F(r, \phi)$.
Furthermore, $\log\sigma_{F_1}(r)=\log\sigma_F(r)-\log|a_m|-m\log r$, whence,
$s_{F_1}(r)=s_F(r)-m$. Therefore, both assumptions and conclusions of Theorems~\ref{thm1} and~\ref{thm2}
remain invariant under this normalization.

Then, we put $F_2(z)=F_1(A_F^{-1}z)$ with $A_F=\max\bigl\{ 2\sum_{k\ge 1}|a_k|, 1 \bigr\}$.
This function already has the form we need, and both assumptions
and conclusions of Theorems~\ref{thm1} and~\ref{thm2} remain invariant under the scaling
$z\mapsto A_F^{-1}z$.

\subsection{Main tools}
Our main tool will be the following lemma:

\begin{lemma}[Log-integrability]\label{lemma:log_integ}
For any $p\ge 1$ and $t > 0$,
\[
\cE \| X_t \|_p^p \le \left( C p \right)^{6 p} \,.
\]
In particular, for $\la\ge 1$,
\[
\bigl( \cP \times m \bigr)\left\{ (\omega, \theta) \in \Omega\times\left[-\pi, \pi \right]\colon\, \left| X_t(\theta) \right| > \lambda \right\} 
\le C \exp(-c \lambda^{1/6}) \,,
\]
where $m$ is the Lebesgue measure on $\left[-\pi, \pi \right]$.
\end{lemma}

The first statement of this lemma is our recent result from~\cite{NNS}. The second statement follows from the first one
by Chebyshev's inequality.

\medskip Our second tool is a version of the classical Jensen formula. The standard version corresponds to
the case $\phi \equiv 1$.

\begin{lemma}[Jensen-type formula]\label{lemma:gen_Jensen_formula}
Let $F$ be an entire function with $F(0)\ne 0$. Then, for any $2\pi$-periodic $C^2$-function $\phi$ and every $R>0$,
we have
\begin{multline*}
\int_0^R \frac{n_F(t, \phi)}{t} \,\dd{t}  =
\int_{-\pi}^\pi \phi(\theta) \bigl[ \log | F(R e^{i \theta})| - \log |F(0)| \bigr]\,\frac{\dd{\theta}}{2 \pi} \\[7pt]
+ \int_0^R \frac{\dd{t}}{t} \int_0^t \frac{\dd{s}}{s}
\int_{-\pi}^\pi \phi^\dprime(\theta) \log | F(s e^{i \theta})| \frac{\dd{\theta}}{2 \pi} \,.
\end{multline*}
\end{lemma}
\begin{remark*}
The repeated integral of the function
\[
s \mapsto \int_{-\pi}^\pi \phi^\dprime(\theta) \log | F(s e^{i \theta})| \frac{\dd{\theta}}{2 \pi}
\]
on the RHS converges absolutely at $s=0,\,t=0$, since for $s\to 0$,
\begin{equation*}
\int_{-\pi}^\pi \phi^\dprime(\theta) \log | F(s e^{i \theta})| \frac{\dd{\theta}}{2 \pi} =
\int_{-\pi}^\pi \phi^\dprime(\theta) \left[ \log | F(s e^{i \theta})| - \log |F(0)| \right] \frac{\dd{\theta}}{2 \pi} = O(s) \,.
\end{equation*}
\end{remark*}

\medskip\noindent{\em Proof of Lemma~\ref{lemma:gen_Jensen_formula}}:
For $C^2$-functions $U,V$ on a bounded domain $G$ with smooth boundary, Green's identity states that
\begin{equation*}
\iint_{G} ( U \lapl V - V \lapl U ) \, \dd{A} =
\int_{\partial G} \left( U \frac{\partial V}{\partial n} - V \frac{\partial U}{\partial n} \right) \dd{S}\,,
\end{equation*}
where $A$ stands for the planar area measure and $S$ for the length.

We set $U(r, \theta) = \frac{1}{2 \pi} \log |F(r e^{i \theta})|$ and $V(r,\theta) = \phi(\theta) \log \frac{R}{r}$. These functions are not in $C^2$, but their singularities can be handled by a standard device:
first, we exclude from the disk $R\,\bar\bD$ $\varepsilon$-neighbourhoods of zeroes of
$F$ and of the origin, then apply Green's formula and let $\varepsilon\to 0$.
The rest is a straightforward computation. \hfill $\Box$

\section{Proof of Theorem~\ref{thm1}}

\subsection{}
By Jensen's formula,
\begin{eqnarray*}
\int_{R_1}^{R_2} \frac{n_F(t)}t\, {\rm d}t &=&
\int_{-\pi}^\pi \bigl[ \log | F(R_2 e^{{\rm i} \theta})| - \log |F(R_1 e^{{\rm i}\theta})| \bigr]\,
\frac{\dd{\theta}}{2 \pi} \\[7pt]
&=& \bigl[ \log\sigma_F(R_2) - \log \sigma_F(R_1) \bigr] +
\int_{-\pi}^\pi \bigl[  X_{R_2}(\theta) - X_{R_1}(\theta) \bigr]\, \frac{\dd{\theta}}{2 \pi} \\[7pt]
&=& \int_{R_1}^{R_2} \frac{s_F(t)}t\, {\rm d}t
+ \int_{-\pi}^\pi \bigl[  X_{R_2}(\theta) - X_{R_1}(\theta) \bigr]\, \frac{\dd{\theta}}{2 \pi}\,.
\end{eqnarray*}
We define the sequence $r_k\uparrow\infty$ so that $s_F(r_k)=k^2$ and put $\delta_k = k^{-1} \log^{-2} k$, $k\ge 2$.
The set
\[
E = [1, r_2e^{\delta_2}] \cup\, \bigcup_{k\ge 3} \left[ r_k e^{-\delta_{k-1}}, r_k e^{\delta_k}\right]
\] will be the exceptional set of finite logarithmic length.
Note that we will be interested only in the intervals $[ r_k, r_{k+1} ]$ whose logarithmic length is
not small: $\log (r_{k+1}/r_k) \ge 2\delta_k$. Otherwise, the whole interval  $[ r_k, r_{k+1} ]$ is contained in
the exceptional set $E$.

\subsection{}
Given $r\in [1, \infty)\setminus E$, we choose $k$ so that $r_k e^{\delta_k} \le r \le r_{k+1}e^{-\delta_k}$. Then
\begin{eqnarray*}
n_F(r) &\le& n_F(r_{k+1}e^{-\delta_k}) \le \frac1{\delta_k} \int_{r_{k+1}e^{-\delta_k}}^{r_{k+1}} \frac{n_F(t)}{t} \,{\rm d} t \\[7pt]
&=& \frac1{\delta_k}\, \int_{r_{k+1}e^{-\delta_k}}^{r_{k+1}} \frac{s_F(t)}{t} \, {\rm d} t \ + \
\frac1{\delta_k} \int_{-\pi}^\pi \bigl[ X_{r_{k+1}}(\theta) - X_{r_{k+1}e^{-\delta_k}}(\theta) \bigr] \,
\frac{\dd{\theta}}{2 \pi} \\[7pt]
&\le& s_F(r_{k+1}) + \frac1{\delta_k} \bigl[ \| X_{r_{k+1}}\|_1  + \| X_{r_{k+1}e^{-\delta_k}} \|_1 \bigr]\,.
\end{eqnarray*}
Similarly,
\[
n_F(r) \ge s_F(r_k) - \frac1{\delta_k} \bigl[ \| X_{r_k}\|_1  + \| X_{r_{k}e^{\delta_k}} \|_1 \bigr]\,.
\]
Combining these bounds and using the monotonicity of the function $s_F$, we get
\begin{multline}\label{eq:ast}
\bigl| n_F(r) - s_F(r) \bigr| \le \bigl[ s_F(r_{k+1}) - s_F(r_k) \bigr] \\[7pt]
+  \frac1{\delta_k} \bigl[ \| X_{r_k}\|_1  + \| X_{r_{k}e^{\delta_k}} \|_1 +
\| X_{r_{k+1}}\|_1  + \| X_{r_{k+1}e^{-\delta_k}} \|_1 \bigr]\,.
\end{multline}
Since $s_F(r_k)=k^2$, we have
$  s_F(r_{k+1}) - s_F(r_k) = 2k+1 $.
Applying H\"older's inequality and then Lemma~\ref{lemma:log_integ}, we see that, for any $r\ge 1$ and any $p<\infty$,
\[
\cE \bigl\{ \| X_r \|_1^p \bigr\} \le \cE \bigl\{ \| X_r \|_p^p \bigr\} \le (Cp)^{6p}\,,
\]
whence
\[
\cP \bigl\{ \| X_r \|_1 > t \bigr\} \le t^{-p} \cE \bigl\{ \| X_r \|_1^p \bigr\}
\le \bigl( t^{-1} \cdot Cp^6 \bigr)^p\,.
\]
Letting $t= e \cdot Cp^6$ and $p=2\log k$, we get
\[
\cP \bigl\{ \| X_r \|_1 > C \log^6 k \} \le \frac1{k^2}\,.
\]
Therefore, by the Borel-Cantelli lemma, for almost every $\omega\in\Omega$,
there exists $k_0(\omega)$ such that, for $k\ge k_0(\omega)$,
\[
\frac1{\delta_k} \bigl[ \| X_{r_k}\|_1  + \| X_{r_{k}e^{\delta_k}} \|_1 +
\| X_{r_{k+1}}\|_1  + \| X_{r_{k+1}e^{-\delta_k}} \|_1 \bigr] = k \log^2k \cdot O(\log^6 k) = O(k \log^8 k).
\]
Hence,
\[
\bigl| n_F(r) - s_F(r) \bigr| \le O(k \log^8 k) = O_\gamma(s_F(r)^\gamma)\,.
\]
This proves the first part of Theorem~\ref{thm1}.

\subsection{}
The proof of the second part (that is, the estimate for $\cE | n_F(r)-s_F(r)|$) is simlar.
Averaging the upper bound~\eqref{eq:ast} and then using Lemma~\ref{lemma:log_integ}, we get
\begin{multline*}
\cE \bigl| n_F(r)-s_F(r) \bigr| \le
\bigl[ s_F(r_{k+1}) - s_F(r_k) \bigr]
+  \frac1{\delta_k} \cE \bigl[ \| X_{r_k}\|_1  + \| X_{r_{k}e^{\delta_k}} \|_1 +
\| X_{r_{k+1}}\|_1  + \| X_{r_{k+1}e^{-\delta_k}} \|_1 \bigr] \\[7pt]
\le 2k+1 + Ck\log^2 k = O_\gamma \bigl( s_F(r)^\gamma \bigr)\,.
\end{multline*}
This completes the proof of Theorem~\ref{thm1} \hfill $\Box$

\subsection{Remark on the notion of ``smallness'' of an exceptional set $E$}
While the notion of smallness we used (finite logarithmic measure) is standard and convenient
for most applications, the proof shows a bit more. Namely, our exceptional set $E$
can be covered by intervals whose logarithmic lengths form a fixed decreasing sequence with a finite sum ($(k\log^2 k)^{-1}$
in our case). Replacing the particular choice of parameters used in the proof of Theorem~\ref{thm1} by a free one,
we can fix an arbitrary increasing  convex sequence $(\la_k)$, $\la_1>1$,
and take the points
$r_k$ so that $s_F(r_k)=\la_k$. Put
\[
\delta_k = \frac{\log^6 (k+1)}{\la_{k+1}-\la_k}\,.
\]
Then, with probability $1$, we get
\[
\bigl| n_F(r) - s_F(r) \bigr| \le C (\la_{k+1}-\la_k)
\quad {\rm for\ } r_k e^{\delta_k} \le r \le r_{k+1} e^{-\delta_k} \ {\rm and\ large\ enough\ } k\,.
\]
Choosing various sequences $\la_k$, we get statements similar to Theorem~\ref{thm1}
in which better control of the exceptional set $E$ can be achieved at the cost of
worse control of the error term. Note that since we cannot control the sequence $r_k$
without any a priori knowledge about the growth of $F$, a result of this type is
meaningful only when $\sum_k \delta_k < \infty$ (otherwise, the exceptional intervals
$[r_k e^{-\delta_{k-1}}, r_ke^{\delta_k}]$ may cover the whole ray $[r_1, +\infty)$).
This forces us to take $\la_k$ of order $k^2$ at the very least. So, Theorem~\ref{thm1}, as stated, is, in a sense, an extremal case.

Also note that the considerations of Section~\ref{subsub:lacunary} show that each result of
this type is essentially sharp up to a factor $\log^6(k+1)$ in the definition of $\delta_k$,
which comes from the Borel-Cantelli estimate.

\section{Several lemmas}

Here, we collect several lemmas needed for the proof of Theorem~\ref{thm2}.
The first lemma is a straightforward corollary to the Jensen-type formula
given in Lemma~\ref{lemma:gen_Jensen_formula}.
\begin{lemma}\label{lemma:corol_to_Jensen}
Let $F$ be a random entire function represented by Rademacher Taylor series with $F(0)\ne 0$.
Then, for any $2\pi$-periodic $C^2$-function $\phi$ and every
$0<R_1<R_2<\infty$, we have
\begin{equation}\label{int_avg_n_F_formula}
\int_{R_1}^{R_2} \frac{\cE n_F(t,\phi)}{t} \, \dd{t} =
\avg{\phi} \int_{R_1}^{R_2} \frac{s_F(t)}{t} \, \dd{t} + \cE \avg{\phi \cdot \left(  X_{R_2} -  X_{R_1} \right)}
+ \int_{R_1}^{R_2} \frac{\dd{t}}{t} \int_0^t \frac{\cE \avg{\phi^\dprime X_s} \dd{s}}{s} \,,
\end{equation}
and
\begin{equation}\label{int_n_F_formula}
\int_{R_1}^{R_2} \frac{n_F(t,\phi)}{t} \, \dd{t} = \int_{R_1}^{R_2} \frac{\cE n_F(t, \phi)}{t} \, \dd{t} +
\avg{\phi \cdot \left( \wbar{X}_{R_2} - \wbar{X}_{R_1}\right)}
+ \int_{R_1}^{R_2} \frac{\dd{t}}{t} \int_0^t \frac{\avg{\phi^\dprime \, \wbar{X}_s} \dd{s}}{s}\,.
\end{equation}
\end{lemma}

\medskip
The next lemma gives an approximation of the Taylor series $F$ by  ``the central group'' of its terms.
We recall that the maximal term and the central index of the Taylor series $F$ are defined  as
\[
\mu _F(r) = \max_{k\ge 0} \bigl\{ |a_k|r^k \bigr\} \quad {\rm and} \quad \nu_F(r) = \max \bigl\{ k \colon |a_k|r^k = \mu_F(r) \bigr\} \,.
\]

\begin{lemma}\label{approx_by_main_terms}
Given $r\ge 1$ and $\tau > 0$, we write $\nu_{-} = \nu_F(r e^{-\tau})$, $\nu_{+} = \nu_F(r e^{\tau})$.
Then
\[
\Bigl| F(z) - \sum_{\nu_{-} \le k \le  \nu_{+}} \xi_k a_k z^k \Bigr| \le
\frac{2 \sigma_F(r)}{e^\tau - 1} \,, \qquad r=|z|\,.
\]
\end{lemma}
\begin{proof}
By the definition of the indices $\nu_\pm$, we have
\[
|a_k|r^k = |a_k| (re^{-\tau})^k e^{\tau k}
\le |a_{\nu_-}| (re^{-\tau})^{\nu_-} e^{\tau k}
= |a_{\nu_{-}}| r^{\nu_{-}} e^{-\tau({\nu_{-}} - k)}\,, \qquad 0\le k \le {\nu_{-}} \,,
\]
and similarly,
$ |a_k|r^k \le |a_{\nu_{+}}| r^{\nu_{+}} e^{-\tau(k - {\nu_{+}})} $ for $k \ge {\nu_{+}}$.
Therefore,
\begin{multline*}
\Bigl( \sum_{0\le k < \nu_{-}} + \sum_{k > \nu_{+}} \Bigr) |a_k| r^k \le
\left( |a_{\nu_{-}}| r^{\nu_{-}} + |a_{\nu_{+}}| r^{\nu_{+}} \right) \frac{1}{e^\tau - 1} \\
 \le  \frac{\sqrt{2}}{e^\tau - 1} \left( |a_{\nu_{-}}|^2 r^{2 \nu_{-}} + |a_{\nu_{+}}|^2 r^{2 \nu_{+}} \right)^{\frac12}
 \le \frac{2 \sigma_F(r)}{e^\tau - 1}
\end{multline*}
proving the lemma.
\end{proof}

\medskip
Our last lemma is a simple application of the Borel-Cantelli Lemma.
\begin{lemma}\label{lemma:BC}
Let $Y_k$ be a sequence of random variables such that, for every $p\ge p_0$,
\begin{equation}\label{cnd_bnd_for_Y_k}
\cE \left| Y_k \right|^p \le \left( G_k(p) \right)^p \,,
\end{equation}
where $p \mapsto G_k(p)$ is a sequence of increasing functions on $\left[ 1, \infty \right)$.
Then, almost surely,
\[ \limsup_{k\to\infty} \frac{|Y_k|}{G_k(\log k)} \le e\,. \]
\end{lemma}
\begin{proof}
Let $\eta > 1, t > 0$. By Chebyshev's inequality,
\[
\cP \{ \left| Y_k \right| >  t \} \le \frac{\cE \left| Y_k \right|^p}{t^p} \le \left( \frac{G_k(p)}{t} \right)^p\,.
\]
Choosing $t = e^\eta G_k(p)$ and $p = \log k$, we see that
\[
\cP \{ \left| Y_k \right| >  e^\eta G_k(\log k) \} \le k^{-\eta}\,.
\]
Then, by the Borel-Cantelli Lemma,
almost surely,
\[ \limsup_{k\to\infty} \frac{|Y_k|}{G_k(\log k)} \le e^\eta\,. \]
Letting $\eta \to 1$, we get the result.
\end{proof}

\section{Proof of Theorem~\ref{thm2}}
The idea of the proof is similar to the one for Theorem~\ref{thm1}: we need to find a sufficiently dense
sequence of ``interpolation points'' $r_k$ where $n_F(r_k, \phi)$ is well approximated by $\avg{\phi} s_F(r_k)$.
The proof of the almost sure bound is significantly more complicated since we have to control the error term
\[
\int_1^r \frac{\avg{\phi^\dprime \, \wbar{X}_s} \dd{s}}{s}\,,
\]
which requires a new idea when the value $s_F(r)$ is comparable to or less than $\log r$.

\medskip
We start by introducing two sequences $r_k \uparrow \infty, k\ge 3$ and $\delta_k \downarrow 0,\, \sum_k \delta_k < \infty$, to be chosen later.
The set
\[
E = [1, r_3] \cup\, \bigcup_{k\ge 3} \left[ r_k e^{-\delta_{k-1}}, r_k e^{\delta_k}\right]
\]
will serve as
our  exceptional set of finite logarithmic length.
Below, we always assume that $\log \frac{r_{k+1}}{r_k} > 2 \delta_k$;
otherwise, the whole interval $[r_k, r_{k+1}]$ is contained in the exceptional set $E$.

\medskip
Till the end of the proof, we fix some $q_0 > 1$ and put
$p_0 = \frac{q_0}{q_0 - 1}$.

\subsection{Preliminary estimates}
We use the following notation:
\[
Q(t) = Q(t; \phi) \eqdef \int_1^t \avg{\phi^\dprime X_s} \frac{\dd{s}}{s},
\quad \wbar{Q}(t) = Q(t) - \cE Q(t) = \int_1^t \avg{\phi^\dprime \overline{X}_s} \frac{\dd{s}}{s}\,.
\]
The next two claims approximate the functions $n_F$ and $\cE n_F$ outside the exceptional set.

\begin{claim}\label{clm_approx_for_n_F}
Suppose that $r \in [r_3, \infty] \setminus E$. Choose $k$ so that $r_k e^{\delta_k} \le r \le r_{k+1} e^{-\delta_k}$. Then
\[
|n_F(r,\phi) - \cE n_F(r,\phi)| \le \left[ s_F(r_{k+1}) - s_F(r_k)\right] + \left| \wbar{Q}(r_k) \right| + \left| \wbar{Q}(r_{k+1}) \right| + {\rm ET}_1 + {\rm ET}_2 \,,
\]
where the error terms ${\rm ET}_1$ and ${\rm ET}_2$ are given by
\begin{multline*}
{\rm ET}_1 = \frac{1}{\delta_k} \bigr[ \| \wbar{X}_{r_k} \|_1 + \| \wbar{X}_{r_k e^{\delta_k}} \|_1 + \| \wbar{X}_{r_{k+1} e^{-\delta_k}} \|_1 + \| \wbar{X}_{r_{k+1}} \|_1 \bigr] \\
+ \| \phi^\dprime \|_{q_0} \Bigl( \int_{r_k}^{r_k e^{\delta_k}} + \int_{r_{k+1} e^{-\delta_k}}^{r_{k+1}} \Bigr) \| \wbar{X}_s \|_{p_0} \frac{\dd{s}}{s} \, ,
\end{multline*}
and
\[
{\rm ET}_2 = C\,\bigl( \tfrac{1}{\delta_k} + \| \phi^\dprime \|_1 \bigr) \,.
\]
\end{claim}

\begin{claim}\label{clm_approx_for_cE_n_F}
Under the assumptions of Claim~\ref{clm_approx_for_n_F}, we have
\[
\cE \bigl| n_F(r,\phi) - \avg{\phi} s_F(r) \bigr| \le  \bigl[ s_F(r_{k+1}) - s_F(r_k) \bigr]
 + C(q_0) \| \phi^\dprime \|_{q_0} \log r_{k+1}  + \frac{C}{\delta_k} \,.
\]
\end{claim}
\begin{proof}[Proof of Claim~\ref{clm_approx_for_n_F}]
By the monotonicity of the function $n_F(r, \phi)$,
\begin{multline*}
n_F(r,\phi) \le  n_F(r_{k+1}e^{-\delta_k} ,\phi) \le \frac{1}{\delta_k} \int_{r_{k+1}e^{-\delta_k}}^{r_{k+1}} \frac{n_F(t,\phi)}{t} \, \dd{t} \\[5pt]
\quad \overset{\eqref{int_n_F_formula}}{=}\frac{1}{\delta_k} \int_{r_{k+1}e^{-\delta_k}}^{r_{k+1}} \frac{\cE n_F(t,\phi)}{t} \, \dd{t} + \frac{1}{\delta_k} \avg{\phi \cdot \left( \wbar{X}_{r_{k+1}} - \wbar{X}_{r_{k+1}e^{-\delta_k}} \right)}
+\,  \frac{1}{\delta_k} \int_{r_{k+1}e^{-\delta_k}}^{r_{k+1}} \frac{\dd{t}}{t} \int_0^t \frac{\dd{s}}{s} \avg{\phi^\dprime\, \wbar{X}_s}\,.
\end{multline*}
Since $0\le \phi \le 1$, the second term on the RHS does not exceed
\[ \frac{1}{\delta_k} \left( \| \wbar{X}_{r_{k+1}} \|_1 + \| \wbar{X}_{r_{k+1}e^{-\delta_k}}\|_1 \right)\,.\]
The third term can be written as
\begin{multline*}
\Bigl( \int_0^1 + \int_1^{r_{k+1}} \Bigr) \avg{\varphi'' \wbar X_s} \frac{\dd{s}}{s}
- \, \frac{1}{\delta_k} \int_{r_{k+1}e^{-\delta_k}}^{r_{k+1}} \frac{\dd{t}}{t}\, \int_t^{r_{k+1}}
\avg{ \phi^\dprime\, \wbar{X}_s } \, \frac{\dd{s}}{s}  \\[5pt]
\le
|\wbar{Q}(r_{k+1})| + \int_0^1 \| \varphi'' \wbar X_s \|_1 \frac{\dd{s}}{s}
+ \frac{1}{\delta_k} \int_{r_{k+1}e^{-\delta_k}}^{r_{k+1}} \frac{\dd{t}}{t}\, \int_t^{r_{k+1}}
\| \phi^\dprime\, \wbar{X}_s \|_1 \, \frac{\dd{s}}{s} \\[5pt]
\le |\wbar{Q}(r_{k+1})|
+ \| \varphi'' \|_1\, \int_0^1 \| \wbar X_s \|_\infty \frac{\dd{s}}{s}
+  \| \varphi'' \|_{q_0}\, \int_{r_{k+1}e^{-\delta_k}}^{r_{k+1}}
\| \wbar{X}_s \|_{p_0} \, \frac{\dd{s}}{s}
\end{multline*}
Due to our normalization of $F$, for $|z|\le 1$, we have
$1-\tfrac12 |z| \le |F(z)| \le 1+\tfrac12 |z|$, whence,
$-|z|\le \log|F(z)| \le \tfrac12 |z|$. We also have
$1 \le \sigma_F(r) \le 1+r$, whence, $0\le \log\sigma_F(r) \le r$.
Thus, for $r=|z|\le 1$, we get
$-2r \le X_r = \log |F| - \log \sigma_F \le \tfrac12 r$,
whence, $|\wbar{X}_r|\le 2.5 r$. Therefore,
\[
\int_0^1 \| \wbar X_s \|_\infty \frac{\dd{s}}{s} \le 2.5\,.
\]
Putting these estimates together, we obtain
\begin{multline}\label{eq:A}
n_F(r,\phi) \le
\frac{1}{\delta_k} \int_{r_{k+1}e^{-\delta_k}}^{r_{k+1}} \frac{\cE n_F(t,\phi)}{t} \, \dd{t}
+ |\wbar{Q}(r_{k+1})| \\[5pt]
+ \frac{1}{\delta_k} \left( \| \wbar{X}_{r_{k+1}} \|_1 + \| \wbar{X}_{r_{k+1}e^{-\delta_k}}\|_1 \right)
+ \| \varphi'' \|_{q_0}\, \int_{r_{k+1}e^{-\delta_k}}^{r_{k+1}}
\| \wbar{X}_s \|_{p_0} \, \frac{\dd{s}}{s}  + 2.5 \| \phi''\|_1\,.
\end{multline}

Next,
\begin{multline*}
n_F(r,\phi) - \cE n_F(r,\phi) \le
\frac{1}{\delta_k} \int_{r_{k+1}e^{-\delta_k}}^{r_{k+1}} \frac{\cE \bigl[ n_F(t,\phi) - n_F(r, \phi)\bigr]}{t} \, \dd{t}
+ |\wbar{Q}(r_{k+1})| \\[5pt]
+ \frac{1}{\delta_k} \left( \| \wbar{X}_{r_{k+1}} \|_1 + \| \wbar{X}_{r_{k+1}e^{-\delta_k}}\|_1 \right)
+ \| \varphi'' \|_{q_0}\, \int_{r_{k+1}e^{-\delta_k}}^{r_{k+1}}
\| \wbar{X}_s \|_{p_0} \, \frac{\dd{s}}{s} + 2.5 \| \phi''\|_1\,.
\end{multline*}
Since $0\le \phi \le 1$, we have for $t\ge r$, $n_F(t,\phi) - n_F(r, \phi) \le n_F(t) - n_F(r)$.
Therefore,
\begin{multline*}
\frac{1}{\delta_k} \int_{r_{k+1}e^{-\delta_k}}^{r_{k+1}} \frac{\cE \bigl[ n_F(t,\phi) - n_F(r, \phi)\bigr]}{t} \, \dd{t}
\le \frac{1}{\delta_k} \int_{r_{k+1}e^{-\delta_k}}^{r_{k+1}} \frac{\cE \bigl[ n_F(t) - n_F(r)\bigr]}{t} \, \dd{t} \\[5pt]
\stackrel{r\ge r_ke^{\delta_k}}\le
\frac1{\delta_k}\,
\Bigl[ \int_{r_{k+1}e^{-\delta_k}}^{r_{k+1}} \frac{\cE n_F(t)}{t} \, \dd{t}
- \int_{r_{k}}^{r_{k}e^{\delta_k}} \frac{\cE n_F(t)}{t} \, \dd{t} \Bigr]\,.
\end{multline*}
Applying relation~\eqref{int_n_F_formula} in Lemma~\ref{lemma:corol_to_Jensen} (with $\phi =1 $) and then
Lemma~\ref{lemma:log_integ}, we see that the RHS equals
\begin{multline*}
\frac1{\delta_k}\,
\Bigl[ \int_{r_{k+1}e^{-\delta_k}}^{r_{k+1}}
- \int_{r_{k}}^{r_{k}e^{\delta_k}} \Bigr]\, \frac{s_F(t)}{t} \, \dd{t}
+ \frac1{\delta_k} \cE \bigl[ \langle X_{r_{k+1}}-X_{r_{k+1}e^{-\delta_k}} \rangle
- \langle X_{r_{k}e^{\delta_k}}-X_{r_{k}} \rangle \bigr] \\[5pt]
\le \bigl[ s_F(r_{k+1}) - s_F(r_k) \bigr]
+ \frac1{\delta_k} \cE \bigl[ \| X_{r_{k+1}} \|_1 + \| X_{r_{k+1}e^{-\delta_k}} \|_1
+ \| X_{r_{k}e^{\delta_k}} \|_1 + \| X_{r_{k}}\|_1 \bigr] \\[7pt]
\le \bigl[ s_F(r_{k+1}) - s_F(r_k) \bigr] + \frac{C}{\delta_k}\,,
\end{multline*}
whence
\begin{multline*}
n_F(r,\phi) - \cE n_F(r,\phi) \le \bigl[ s_F(r_{k+1}) - s_F(r_k) \bigr]
+ |\wbar{Q}(r_{k+1})| \\[7pt]
+  \frac{1}{\delta_k} \left( \| \wbar{X}_{r_{k+1}} \|_1 + \| \wbar{X}_{r_{k+1}e^{-\delta_k}}\|_1 \right)
+ \| \varphi'' \|_{q_0}\, \int_{r_{k+1}e^{-\delta_k}}^{r_{k+1}}
\| \wbar{X}_s \|_{p_0} \, \frac{\dd{s}}{s} + 2.5 \| \phi'' \|_1 + \frac{C}{\delta_k}\,.
\end{multline*}
The proof of the matching lower bound
\begin{multline*}
n_F(r,\phi) - \cE n_F(r,\phi) \ge -\bigl[ s_F(r_{k+1}) - s_F(r_k) \bigr] - |\wbar{Q}(r_{k})| \\[7pt]
-  \frac{1}{\delta_k} \left( \| \wbar{X}_{r_{k}} \|_1 + \| \wbar{X}_{r_{k}e^{\delta_k}}\|_1 \right)
- \| \varphi'' \|_{q_0}\,  \int_{r_{k}}^{r_{k}e^{\delta_k}}
\| \wbar{X}_s \|_{p_0} \, \frac{\dd{s}}{s} - 2.5 \| \phi'' \|_1 - \frac{C}{\delta_k}
\end{multline*}
is very similar and we skip it. \end{proof}

\begin{proof}[Proof of Claim~\ref{clm_approx_for_cE_n_F}]
The proof is similar to the previous one. We estimate the first term on the RHS of the upper bound~\eqref{eq:A}
applying  Lemmas~\ref{lemma:corol_to_Jensen} and~\ref{lemma:log_integ}:
\begin{multline*}
\frac{1}{\delta_k} \int_{r_{k+1}e^{-\delta_k}}^{r_{k+1}} \frac{\cE n_F(t,\phi)}{t} \, \dd{t}
\le \langle \phi \rangle\, \frac{1}{\delta_k} \int_{r_{k+1}e^{-\delta_k}}^{r_{k+1}} \frac{s_F(t)}{t} \, \dd{t} \\[7pt]
+ \frac1{\delta_k} \cE \bigl[ \| X_{r_{k+1}e^{-\delta_k}}\|_1 + \| X_{r_{k+1}} \|_1\bigr]
+ \frac1{\delta_k} \int_{r_{k+1}e^{-\delta_k}}^{r_{k+1}} \frac{\dd t}{t}\,
\int_0^t \cE \| \phi'' X_s \|_1 \, \frac{\dd s}{s} \\[7pt]
\le \langle \phi \rangle\, s_F(r_{k+1}) + \frac{C}{\delta_k}
+ \| \phi''\|_{q_0} \Bigl[ \int_0^1 + \int_1^{r_{k+1}} \Bigr] \cE \| X_s \|_{p_0}\, \frac{\dd s}{s} \\[7pt]
\le  \langle \phi \rangle\, s_F(r_{k+1}) + \frac{C}{\delta_k}
+ C(q_0) \| \phi''\|_{q_0} \log r_{k+1}\,.
\end{multline*}
Plugging this estimate into~\eqref{eq:A}, we obtain
\begin{multline*}
n_F(r, \phi) - \langle \phi \rangle s_F(r)
\le \bigl[ s_F(r_{k+1}) - s_F(r_k) \bigr] + \frac{C}{\delta_k}
+ \| \phi''\|_{q_0} C(q_0) \log r_{k+1}
+ |\wbar{Q}(r_{k+1})| \\[7pt]
+ \frac{1}{\delta_k} \left( \| \wbar{X}_{r_{k+1}} \|_1 + \| \wbar{X}_{r_{k+1}e^{-\delta_k}}\|_1 \right)
+ \| \varphi'' \|_{q_0}\, \int_{r_{k+1}e^{-\delta_k}}^{r_{k+1}}
\| \wbar{X}_s \|_{p_0} \, \frac{\dd{s}}{s}  + 2.5 \| \phi''\|_1\,.
\end{multline*}
Combining with the matching lower bound and taking the expectation, we get
\[
\cE \bigl| n_F(r, \phi) - \langle \phi \rangle s_F(r)  \bigr| \le
\bigl[ s_F(r_{k+1}) - s_F(r_k) \bigr] + \frac{C}{\delta_k}
+ C(q_0) \| \phi''\|_{q_0}  \log r_{k+1}\,,
\]
proving Claim~\ref{clm_approx_for_cE_n_F}. \end{proof}

\subsection{Estimate of $\cE | n_F(r, \phi) - \langle \phi \rangle s_F(r) |$}\label{subsect:mean_estimate}

Using Claim~\ref{clm_approx_for_cE_n_F}, we readily prove assertion (ii) of Theorem~\ref{thm2}.
\begin{proof}
We need to estimate the expression
\[
\left[ s_F(r_{k+1}) - s_F(r_k) \right]
 + C(q_0) \| \phi^\dprime \|_{q_0} \log r_{k+1}  + \frac{C}{\delta_k}\,,
\]
which appears on the RHS of the bound given in Claim~\ref{clm_approx_for_cE_n_F}.
We choose the sequence $r_k$ so that $s_F(r_k)+\log r_k = k^2$. Then
\[
s_F(r_{k+1})-s_F(r_k) \le 3k \le 3 \bigl( s_F(r_k)^{\frac12} + \log^{\frac12} r_k \bigr)
< 3 \bigl( s_F(r)^{\frac12} + \log^{\frac12} r \bigr),
\]
and
\[
\log r_{k+1} < 3k + \log {r_k} < 3 s_F(r)^{\frac12} + 4 \log r + 3\,.
\]
Put
$\delta_k = \bigl( k \log^2 k \bigr)^{-1}$. Then for $\gamma>\tfrac12$, we have
\[
\delta_k^{-1} = k \log^2 k
< C(\gamma) \bigl( s_F(r_k)^\gamma  + \log^\gamma r_k\bigr)
< C(\gamma) \bigl( s_F(r)^\gamma  + \log r \bigr)\,,
\]
whence
\[
\cE \bigl| n_F(r,\phi) - \avg{\phi} s_F(r) \bigr| <
C(q_0, \gamma) (1+\|\phi^\dprime \|_{q_0})\, \bigl( s_F(r)^\gamma  + \log r \bigr)\,,
\]
completing the proof.
\end{proof}

\medskip
Now, we start proving the more difficult part (i) of Theorem~\ref{thm2}, that is, the almost sure
estimate for $ | n_F(r, \varphi) - \cE n_F(r, \varphi) |$. For this, we need to estimate the RHS of the
bound given in Claim~\ref{clm_approx_for_n_F}.

\subsection{Easy error terms}
Here we give an almost sure bound for the error terms ${\rm ET}_1$ in Claim~\ref{clm_approx_for_n_F}. Recall that
\begin{multline*}
{\rm ET}_1 =  \frac{1}{\delta_k} \left( \| \wbar{X}_{r_k} \|_1 + \| \wbar{X}_{r_k e^{\delta_k}} \|_1 + \| \wbar{X}_{r_{k+1} e^{-\delta_k}} \|_1 + \| \wbar{X}_{r_{k+1}} \|_1 \right) \\[7pt]
+  \| \phi^\dprime \|_{q_0} \Bigl( \int_{r_k}^{r_k e^{\delta_k}} + \int_{r_{k+1} e^{-\delta_k}}^{r_{k+1}} \Bigr) \| \wbar{X}_s \|_{p_0} \frac{\dd{s}}{s} \,.
\end{multline*}

\begin{claim}
For almost every $\omega \in \Omega$, there exists $k_0 = k_0(\omega)$ such that
\[
{\rm ET}_1 \le \frac{C}{\delta_k} \log^6 k + C(q_0) \| \phi^\dprime \|_{q_0} \qquad\text{for all\ } k \ge k_0 \,.
\]
\end{claim}
\begin{proof}
Let $\rho_k$ be one of the values $r_k, r_k e^{\delta_k}, r_{k+1} e^{-\delta_k}, r_{k+1}$. Then, by Lemma~\ref{lemma:log_integ},
\[
\cE \| \wbar{X}_{\rho_k} \|_1^p \le 2^{p} \cE \| X_{\rho_k} \|_p^p \le (C p )^{6 p} \,.
\]
Hence, by Lemma~\ref{lemma:BC}, for almost every $\omega \in \Omega$ and for every $k \ge k_0(\omega)$, $\| \wbar{X}_{\rho_k} \|_1 \le C \log^6 k$.
Next, by Lemma~\ref{lemma:log_integ},
\[
\cE \Bigl\{ \Bigl( \int_{r_k}^{r_k e^{\delta_k}} + \int_{r_{k+1} e^{-\delta_k}}^{r_{k+1}} \Bigr) \| \wbar{X}_s \|_{p_0} \frac{\dd{s}}{s} \Bigr\}
\le C(q_0) \delta_k \,.
\]
Recalling that $\sum_k \delta_k < \infty$ and applying Chebyshev's inequality and the Borel-Cantelli Lemma, we see that for almost every $\omega \in \Omega$, these
two integrals do not exceed $C(q_0)$ for every $k \ge k_0(\omega)$. This proves the claim.
\end{proof}

\subsection{Crude estimate of the integral $\wbar{Q}(r_{k+1})$}
It remains to estimate the integral
\[
\wbar{Q} = \wbar{Q}(r_{k+1}) = \int_1^{r_{k+1}} \frac{\langle \phi''\, \wbar{X}_t\rangle}{t} \, \dd{t} \,.
\]
We start with a crude bound.
\begin{claim}\label{clm_simple_est_for_Q_bar}
For almost every $\omega \in \Omega$, there exists $k_0 = k_0(q_0, \omega)$ such that
\[
\left| \wbar{Q}(r_{k+1}) \right| \le
C \| \phi^\dprime \|_{q_0} \log^6 k \cdot \log r_{k+1} \,, \qquad k \ge k_0\,.
\]
\end{claim}
\begin{proof}
For any $p \ge p_0$ we have
$ \left| \avg{\phi^\dprime X_t} \right| \le \| \phi^\dprime \|_{q_0} \| X_t \|_p $.
Thus, applying first H\"older's inequality and then Lemma~\ref{lemma:log_integ}, we get
\[
\cE \left| \wbar{Q} \right|^p \le  \| \phi'' \|_{q_0}^p \cdot
\int_1^{r_{k+1}} \cE \| X_t \|_p^p\, \frac{\dd{t}}{t} \cdot (\log r_{k+1})^{p-1} 
\le  \| \phi^\dprime \|_{q_0}^p \cdot \bigl( C p \bigr)^{6p} \cdot \log^p r_{k+1}\,.
\]
Now Lemma~\ref{lemma:BC} yields the required result.
\end{proof}

\subsection{Refined estimate of the integral $\wbar{Q}(r_{k+1})$}

Here, we will present a more delicate estimate for $\wbar{Q}$, which refines the previous one.
The idea is to partition the interval $[1, r_{k+1}]$ into intervals of equal logarithmic length $\tau_k$ with
$1 \ll \tau_k \ll \log r_{k+1}$ and represent $\wbar{Q}$ as a sum of integrals over these intervals. It turns out that these integrals
can be well approximated by independent bounded random variables with zero mean.
Then the natural cancellation in their sum
yields an improved bound for $ \wbar{Q}(r_{k+1}) $.

Put $Z = Z(t) \eqdef \avg{\phi^\dprime X_t}$ and $\wbar{Z} = Z - \cE Z$. Then
\[
\wbar{Q} = \wbar{Q}(r_{k+1}) = \int_1^{r_{k+1}} \frac{\wbar{Z}(t)}{t} \, \dd{t} \,.
\]
We are going to estimate $\cE \left| \wbar{Q} \right|^p$. This will be done in several steps.

\subsubsection{Truncation of the logarithm}
Fix $k$ and $\Lambda = \Lambda(k, r_{k+1})$, and put
\[
\log_\Lambda x = \begin{cases}
\log x ,   & \quad |\log x| \le \Lambda^6\,, \\[5pt]
-\Lambda^6 , & \quad \log x < -\Lambda^6\,, \\[5pt]
\Lambda^6 ,  & \quad \log x > \Lambda^6 \,, \\[5pt]
\end{cases}
\]
and
$ Z_\Lambda(t) = \avg{\phi^\dprime \log_\Lambda | \what{F}_t |} $.
Then,
\[
| Z(t) - Z_\Lambda(t)| \le
\| \phi^\dprime \|_{q_0}  \cdot \bigl\| \log |\what{F}_t| - \log_\Lambda |\what{F}_t|\, \bigr\|_{p_0} \,,
\]
and, for $p \ge p_0$,
\begin{eqnarray*}
\cE \bigl\{ | Z(t) - Z_\Lambda(t)|^p \bigr\}
& \le & \| \phi^\dprime \|_{q_0}^p\,
\cE \bigl\{ \bigl\| \log |\what{F}_t| - \log_\Lambda |\what{F}_t|\, \bigr\|_p^p \bigr\} \\[10pt]
& \le & \| \phi^\dprime \|_{q_0}^p\, \cE \bigl\{ \avg{\indf{E_{\Lambda}(t)}  | X_t |^p}  \bigr\} \,,
\end{eqnarray*}
where $E_\Lambda(t) = \left\{ \theta \in \left[-\pi, \pi \right] \,\colon\, \left| X_t(\theta) \right| > \Lambda^6 \right\}$,
and $\indf{E_{\Lambda}(t)}$ is the indicator function of the set $E_\Lambda (t)$.
Using the Cauchy-Schwarz inequality and Lemma~\ref{lemma:log_integ}, we get
\begin{multline*}
\cE \bigl\{ | \wbar{Z}(t) - \wbar{Z}_\Lambda(t)|^p \bigr\} \le
\cE \bigl\{ | Z(t) - Z_\Lambda(t)|^p \bigr\} \\[10pt]
\le  \| \phi^\dprime \|_{q_0}^p\,
\sqrt{\cE \bigl\{ m_\theta ( E_\Lambda(t) ) \bigr\} \cdot \cE \| X_t \|_{2p}^{2p} }
\le   \| \phi^\dprime \|_{q_0}^p \cdot C \exp(-c \Lambda) \cdot (Cp^6)^p\,,
\end{multline*}
where $\wbar{Z}_\Lambda = Z_\Lambda - \cE Z_\Lambda$ and $m_\theta$ is Lebesgue measure on $[-\pi, \pi]$. Then
\begin{equation}\label{eq:est_for_error_in_cutoff}
\cE \bigl\{ \left| \wbar{Q} - \wbar{Q}_\Lambda \right|^p \bigr\}
= \cE \Bigl\{ \Bigl| \int_1^{r_{k+1}} \left( \wbar{Z}(t) - \wbar{Z}_\Lambda(t) \right) \, \frac{\dd{t}}{t} \Bigr|^p \Bigr\}
\le \left( C p^6 \| \phi^\dprime \|_{q_0} \log r_{k+1} \right)^p e^{-c \Lambda } \,.
\end{equation}

\subsubsection{Replacing the Taylor series $F$ by a group of its central terms}
Let $\tau = \tau (k, r_{k+1})$ be a large parameter (to be chosen later). Let
\[
\what{P}(z) \eqdef \frac{1}{\sigma_F(r)} \sum_{\ell=\nu_F(r e^{-\tau})}^{\nu_F(r e^\tau)} \xi_\ell a_\ell z^\ell \,,
\quad r = |z|\,,
\]
$ \what{P}_r(\theta) = \what{P}(re^{{\rm i}\theta}) $,
$ \wbar{Z}_\Lambda^{\, \tt c.t.} (t) = \avg{\phi^\dprime \bigl(\log_\Lambda |\what{P}_t|
- \cE \log_\Lambda |\what{P}_t| \bigr)} $,
and
\[
\wbar{Q}_\Lambda^{\, \tt c.t.} = \wbar{Q}_\Lambda^{\, \tt c.t.} (r_{k+1}) =
\int_1^{r_{k+1}} \wbar{Z}_\Lambda^{\, \tt c.t.}(t)\, \frac{\dd{t}}t\,.
\]
As before, $\nu_F(r)$ denotes the central index of the Taylor series $F$.

Applying Lemma~\ref{approx_by_main_terms} and using the fact that
\[
|\log_\Lambda x - \log_\Lambda y| \le e^{\Lambda^6} |x - y|\, \quad x,y>0 \,,
\]
we get
\[
\sup_{t \ge 1} \bigl| \wbar{Z}_\Lambda(t) - \wbar{Z}_\Lambda^{\, \tt c.t.}(t) \bigr| \le
C e^{\Lambda^6 - \tau} \| \phi^\dprime \|_1 \,,
\]
whence, for every $\omega \in \Omega$,
\begin{equation}\label{Q_bar_est_for_error_in_poly_approx}
\bigl| \wbar{Q}_\Lambda - \wbar{Q}_\Lambda^{\, \tt c.t.} \bigr|  =
\Bigl| \int_1^{r_{k+1}} \left( \wbar{Z}_\Lambda(t) - \wbar{Z}_\Lambda^{\, \tt c.t.}(t) \right) \, \frac{\dd{t}}{t} \Bigr|
\le C \| \phi^\dprime \|_1 \, \exp(\Lambda^6 - \tau ) \cdot \log r_{k+1}\,.
\end{equation}

\subsubsection{Fast and slow intervals}
From now on, we assume that $1 \ll \tau \ll \log r_{k+1}$ and that $L\stackrel{\rm def}=\tau^{-1} \log r_{k+1}$ is an integer.
For any integer $j$,
consider the intervals $J_j = \bigl[ e^{j\tau}, e^{(j+1)\tau} \bigr]$ of equal logarithmic length $\tau$.
We call the interval $J_j$ taken from this collection {\em slow}
if the central index $\nu_F$ remains constant on $J_j$ as well as on its two neighbouring intervals, that is, if
$ \nu_F\bigl( e^{(j-1)\tau} \bigr) = \nu_F\bigl( e^{(j+2)\tau} \bigr) $.
Otherwise, the interval $J_j$ is called {\em fast}.

On every slow interval $J_j$ the sum $\what{P}$ consists of a single term
\[
\what{P}(z) = \frac{\xi_{\nu_j}a_{\nu_j}z^{\nu_j}}{\sigma_F(|z|)}\,,
\]
where $\nu_j$ is the common value of $\nu_F$ on $J_j$, and therefore
\[
|\what{P}_t| = \frac{|a_{\nu_j}|t^{\nu_j}}{\sigma_F(t)}
\]
is non-random. Hence, for such $t$'s, $\wbar{Z}_\Lambda^{\, \tt c.t.} (t) = 0$;
i.e., slow intervals do not contribute to the integral $\wbar{Q}_\Lambda^{\, \tt c.t.}$. Thus,
\begin{equation}\label{eq:Q-bar-lambda-P}
\wbar{Q}_\Lambda^{\, \tt c.t.} =
\int_0^\tau \Bigl( \sum_{j\in\mathfrak J} \wbar{Z}_\Lambda^{\, \tt c.t.} (e^{j\tau + s}) \Bigr) \, \dd{s} \,,
\end{equation}
where $\mathfrak J$ is the set of indices $j$ such that $J_j\subset [1, \log r_{k+1}]$
and $J_j$ is fast.

\subsubsection{Contribution of fast intervals}
We split the set $\mathfrak J$ into a bounded number of disjoint subsets $\mathfrak J'\subset \mathfrak J$ so
that, for $j_1, j_2\in\mathfrak J'$ and $j_1\ne j_2$, the intervals
\[
\bigl[ \nu_F\bigl( e^{(j_1-1)\tau}\bigr), \, \nu_F\bigl( e^{(j_1+2)\tau}\bigr) \bigr]\,,
\quad
\bigl[ \nu_F\bigl( e^{(j_2-1)\tau}\bigr), \, \nu_F\bigl( e^{(j_2+2)\tau}\bigr) \bigr]
\]
are disjoint (it is easy to see that six subsets $\mathfrak J'$ suffice).
Given $s\in [0, \tau]$, the random variable $ \wbar{Z}_\Lambda^{\, \tt c.t.}(\exp(j \tau + s)) $ may depend
only on $\xi_\ell$ with $\nu_F(e^{(j-1)\tau}) \le \ell \le \nu_F(e^{(j+2)\tau})$. Therefore,
given a subset $\mathfrak J'$ and a value $s\in [0, \tau]$,
the random variables $ \bigl\{ \wbar{Z}_\Lambda^{\, \tt c.t.}(\exp(j \tau + s)) \bigr\}_{j\in\mathfrak J'} $
are independent. This observation allows us to
estimate $\cE |K(s, \mathfrak J')|^p$, where
\[
K(s, \mathfrak J') \stackrel{\rm def}=
\sum_{j\in\mathfrak J'} \wbar{Z}_\Lambda^{\, \tt c.t.}(\exp(j \tau + s))\,.
\]
Indeed, recalling that $\cE \wbar{Z}_\Lambda^{\, \tt c.t.}(t) = 0$
and that $ \bigl| \wbar{Z}_\Lambda^{\, \tt c.t.} (t) \bigr| \le 2 \Lambda^6 \| \phi^\dprime \|_1$,
and applying the classical Khinchin-Marcinkiewicz--Zygmund inequality, we get
\[
\cE \left| K(s, \mathfrak J') \right|^p
\le (C p )^{p/2} \left( 2 \Lambda^6 \| \phi^\dprime \|_1 \right)^p \cdot | \mathfrak J'| ^{p/2}\,.
\]
Since $|\mathfrak J'| \le L = \tau^{-1}\log r_{k+1}$, the RHS does not exceed
$ \bigl( C \Lambda^6 \| \phi^\dprime \|_1 \sqrt{p\cdot \tau^{-1}\log r_{k+1}}\,  \bigr)^p $.

At last, using Minkowski's integral inequality and
recalling that we use only a bounded number of subsets $\mathfrak J'$,
we obtain
\begin{multline}\label{eq:5*}
\cE \bigl| \wbar{Q}_\Lambda^{\, \tt c.t.} \bigr|^p =  \cE \Bigl| \int_0^{\tau}
\sum_{\mathfrak J'}  K(s, \mathfrak J') \, \dd{s} \Bigr|^p
\le
C^p
\Bigl( \int_0^{\tau} \sum_{\mathfrak J' } \left( \cE \left| K(s, \mathfrak J') \right|^p \right)^{1/p} \, \dd{s} \Bigr)^p
\\[7pt]
\le (C \tau)^p \cdot \bigl( C \Lambda^6 \| \phi^\dprime \|_1 \sqrt{p\cdot \tau^{-1}\log r_{k+1}}\,  \bigr)^p
= \bigl( C \| \phi^\dprime \|_1\, \Lambda^6\, \sqrt{p\cdot \tau \log r_{k+1}}\, \bigr)^p\,.
\end{multline}

\subsubsection{Final estimate of $\wbar{Q}$}
Here, we prove the following estimate:
\begin{claim}\label{claim:final_bar-Q}
For a.e. $\omega\in\Omega$, every $k\ge k_0(\omega)$, and every $\varepsilon >0$,
we have
\begin{equation}\label{eq:final_bar-Q}
\bigl| \wbar{Q} \bigr|
\le C(\varepsilon) \| \phi'' \|_{q_0} \bigl( ( \log r_{k+1} )^{\frac12 +\varepsilon} + k^\varepsilon \bigr)\,.
\end{equation}
\end{claim}

\begin{proof}
We assume that $\log r_{k+1} \gg \log^7 k$.
Otherwise, the crude bound from
Claim~\ref{clm_simple_est_for_Q_bar} yields
\[
\bigl| \wbar{Q} \bigr| \le C \| \phi'' \|_{q_0} \, \log^{13}k \,,
\]
which immediately gives us~\eqref{eq:final_bar-Q}.

\medskip
Combining our estimates~\eqref{eq:est_for_error_in_cutoff}, \eqref{Q_bar_est_for_error_in_poly_approx}, and~\eqref{eq:5*},
we get
\begin{multline*}
\bigl( \cE | \wbar{Q} |^p \bigr)^{1/p}
\le \bigl[ \bigl( \cE | \wbar{Q} - \wbar{Q}_\Lambda |^p \bigr)^{1/p}
+ \bigl( \cE | \wbar{Q}_\Lambda - \wbar{Q}_\Lambda^{\, \tt c.t.} |^p \bigr)^{1/p}
+ \bigl( \cE | \wbar{Q}_\Lambda^{\, \tt c.t.} |^p )^{1/p}  \bigr] \\[7pt]
\le C \| \phi'' \|_{q_0}\,
\Bigl[ e^{-c\Lambda/p}\, p^6\, \log r_{k+1} + e^{\Lambda^6 - \tau} \log r_{k+1}
+ \Lambda^6 \sqrt{p \cdot \tau\log r_{k+1}}  \Bigr]\,.
\end{multline*}
Then, applying Lemma~\ref{lemma:BC}, we see that,
for almost every $\omega\in\Omega$ and every $k\ge k_0(\omega)$,
\[
\bigl| \wbar{Q} \bigr| \le
C\, \| \phi'' \|_{q_0} \Bigl[
e^{-c \Lambda/\log k} \cdot \log^6 k \cdot \log r_{k+1}
+ e^{\Lambda^6 - \tau} \cdot \log r_{k+1}
+ \Lambda^6   \, \sqrt{\tau  \log r_{k+1} \cdot \log k}
\Bigr].
\]

\medskip
Now it is time to choose the values of the parameters $\Lambda$ and $\tau$.
We put
\[
\Lambda = C_1 \log k \bigl( \log\log r_{k+1} + \log \log k \bigr)
\quad {\rm and\ then} \quad
\tau = \Lambda^6 + \log\log r_{k+1}
\]
with a sufficiently large constant $C_1$.
Recall that our derivation of the  bound for $ \cE | \wbar{Q}_\Lambda^{\, \tt c.t.} |^p $ used the
condition $1 \ll \tau \ll \log r_{k+1}$ which is guaranteed by the
assumption $\log^7k\!~\ll\!~\log r_{k+1}$.

\medskip
The choice of the parameters $\Lambda$ and $\tau$ yields boundedness of the terms
\[
e^{-c \Lambda/\log k} \cdot \log^6 k \cdot \log r_{k+1}\,, \quad
e^{\Lambda^6 - \tau} \cdot \log r_{k+1}\,.
\]
Thus, it remains to estimate the term
$ \Lambda^6   \, \sqrt{\tau  \log r_{k+1} \cdot \log k} $.
Observe that, for sufficiently large $k$, both $\Lambda^6$ and $\tau$ do not exceed
$ ( \log k )^C + ( \log\log r_{k+1} )^C $.
This yields the estimate
\[
\Lambda^6   \, \sqrt{\tau  \log r_{k+1} \cdot \log k}
\le C(\varepsilon) \bigl( (\log r_{k+1})^{\frac12 + \varepsilon} + ( \log k)^{C(\varepsilon)}  \bigr)
\le
C(\varepsilon) \bigl( (\log r_{k+1})^{\frac12 + \varepsilon} + k^\varepsilon  \bigr)\,,
\]
proving the claim.
\end{proof}

\subsection{Completing the proof of Theorem~\ref{thm2}}
We need to prove the almost sure part \textrm{(i)} of the theorem.
Returning to Claim~\ref{clm_approx_for_n_F}, and plugging in the estimates of all error terms,
for $r_ke^{\delta_k} \le r \le r_{k+1}e^{-\delta_k}$, $k\ge k_0(\omega)$,
we get
\begin{multline*}
|n_F(r,\phi) - \cE n_F(r,\phi)| \le  \bigl( s_F(r_{k+1}) - s_F(r_k) \bigr)
+ \frac{C}{\delta_k} \log^6 k + \\[5pt]
+ C(q_0, \varepsilon) \| \phi'' \|_{q_0} \bigl( ( \log r_{k+1} )^{\frac12 +\varepsilon} +  k^\varepsilon
\bigr)\,.
\end{multline*}
It remains to show  that with the same choice of the parameters $\delta_k$ and $r_k$ as in
Section~\ref{subsect:mean_estimate}, we get the desired result.
First, the choice  $ \delta_k = \bigl( k \log^2 k \bigr)^{-1} $
yields that the RHS of the previous estimate is
\[
\le \bigl( s_F(r_{k+1}) - s_F(r_k) \bigr) + C(\varepsilon) k^{1+\varepsilon} +
C(q_0, \varepsilon)\, \| \phi^\dprime \|_{q_0}\,
\bigl(  ( \log r_{k+1} )^{\frac12 +\varepsilon} + k^\varepsilon \bigr)\,.
\]
At last, we take $r_k$ so that
$ s_F(r_k) + \log r_k  = k^2 $.
Repeating the estimates from Section~\ref{subsect:mean_estimate}, we have
\[
s_F(r_{k+1})-s_F(r_k) \le 3 \bigl( s_F(r)^\frac12 + (\log r)^\frac12 \bigr)
\]
and
\[
k^\varepsilon \le s_F(r)^{\frac12 \varepsilon} + ( \log r )^{\frac12 \varepsilon}\,,
\quad k^{1+\varepsilon} \le s_F(r)^{\frac12(1+\varepsilon)} + ( \log r )^{\frac12 (1+\varepsilon)}\,.
\]
In addition,
\[
( \log r_{k+1} )^{\frac12 + \varepsilon} \le (k+1)^{1+2\varepsilon}
< 4 \bigl( s_F(r)^{\frac12 +\varepsilon} + ( \log r )^{\frac12 + \varepsilon} \bigr)\,.
\]
Therefore, for $k>k_0(\omega)$, we have
\[
|n_F(r,\phi) - \cE n_F(r,\phi)| \le
C(q_0, \varepsilon)\, \bigl( 1+\| \phi'' \|_{q_0}\bigr)\, \bigl(  s_F(r)^{\frac12 +\varepsilon} + ( \log r )^{\frac12 + \varepsilon} \bigr)\,.
\]
Taking $0<\varepsilon<\gamma-\tfrac12$, we finish off the proof of Theorem~\ref{thm2}.
\hfill $\Box$.


\begin{thebibliography}{A}

\bibitem{BNS}{\sc A. Borichev, A. Nishry, M. Sodin},
Entire functions of exponential type represented by pseudo-random and random Taylor series.
J. d'Analyse Math., to appear.
{\tt arXiv:1409.2736}.

\bibitem{Favorov-90} {\sc S. Yu. Favorov},
Growth and distribution of the values of holomorphic mappings of a finite-dimensional
space into a Banach space. Siberian Math. J. {\bf 31} (1990), 137--146.

\bibitem{Favorov-94} {\sc S. Yu. Favorov},
On the growth of holomorphic mappings from a finite-dimensional space into a Banach space.
Mat. Fiz. Anal. Geom. {\bf 1} (1994), 240--251.

\bibitem{Hayman} {\sc W. K. Hayman},
Subhamronic functions, vol. 2. Academic Press, 1989.

\bibitem{KZ} {\sc Z. Kabluchko, D. Zaporozhets}, Asymptotic distribution of complex zeros
of random analytic functions,  Ann. Probab. {\bf 42} (2014), 1374--1395.
{\tt arXiv:1205.5355}

\bibitem{L-O} {\sc J. E. Littlewood, A. C. Offord},
On the distribution of zeros and $a$-values
of a random integral function (II), Ann. of Math. (2) {\bf 49}, (1948), 885--952;
errata {\bf 50} (1949), 990--991.

\bibitem{MaFi1} {\sc M. P. Mahola, V. P. Filevich},
The angular distribution of zeros of random analytic functions,
Ufa Math. J. {\bf 12:4} (2012), 122--135. (Russian)

\bibitem{MaFi2} {\sc M. P. Mahola, V. P. Filevich},
The angular distribution of the values of analytic and random analytic functions,
Mat. Stud. {\bf 38:2} (2012), 147--153.

\bibitem{NNS} {\sc F. Nazarov, A. Nishry, M. Sodin},
Log-integrability of Rademacher Fourier series, with applications to random analytic functions,
Algebra \& Analysis {\bf 25:3} (2013),  147--184. {\tt  arXiv:1301.0529}



\bibitem{Offord-1965} {\sc A. C. Offord},
The distribution of the values of an entire
function whose coefficients are independent random variables. (I)
Proc. London Math. Soc. (3) {\bf 14a} (1965), 199--238.

\bibitem{Offord-1968} {\sc A. C. Offord},
The distribution of zeros of power series whose coefficients are independent random variables.
Indian J. Math. {\bf 9} (1967), 175--196.


\bibitem{Offord-1995} {\sc A. C. Offord},
{The distribution of the values of an entire function whose coefficients
are independent random variables}. (II).
 Math. Proc. Cambridge Phil. Soc. {\bf 118} (1995), 527--542.

\bibitem{Ullrich-88a} {\sc D. C. Ullrich},
{An extension of the Kahane-Khinchine inequality in a Banach space}.
Israel J. Math. {\bf 62} (1988), 56--62.

\bibitem{Ullrich-88b} {\sc D. C. Ullrich},
{Khinchin's inequality and the zeros of Bloch functions}.
Duke Math. J. {\bf 57} (1988), 519--535.

\end{thebibliography}
\end{document}